\documentclass[a4paper,twoside]{article}

1
\def\shorttitle{Unique Determination of General Source Terms}
\def\shortauthor{Y. Kian, Y. Liu and M. Yamamoto}
\def\correction2

\usepackage{amsfonts}
\usepackage{amsmath}
\usepackage{amssymb}
\usepackage{amscd}
\usepackage{dsfont}
\usepackage{amsbsy}
\usepackage{amsthm}
\usepackage{bm}
\usepackage{empheq}
\usepackage{graphicx}
\usepackage{paralist}
\usepackage{psfrag}
\usepackage{color}
\usepackage{cite}
\usepackage{multicol}
\usepackage{subfigure}
\usepackage[normalem]{ulem}

\allowdisplaybreaks[4]

\newfont{\myfnt}{cmssi10 scaled 1440}
\numberwithin{equation}{section}
\catcode`@=11
\def\ps@nk{\def\@oddhead{\vbox{\hbox to \hsize{\pic \footnotesize \it \shorttitle
\hfill \rm \thepage} \vspace{1mm} \vspace*{-2mm}}}
\def\@evenhead{\vbox{\hbox to \hsize{\pic \footnotesize \rm \thepage \hfill \it \shortauthor}
\vspace{1mm} \vspace*{-2mm}}}
\def\@oddfoot{} \def\@evenfoot{}}
\def\ps@first{\def\@oddhead{\vbox{\hbox to \hsize{\pic \footnotesize
} \break}}
\def\@oddfoot{} \def\@evenfoot{}}
\newtheoremstyle{thmstyle}% name
  {6pt}%      Space above
  {6pt}%      Space below
  {\it}%         Body font
  {}%         Indent amount (empty = no indent,\parindent = para indent)
  {\bf}% Thm head font
  {}%        Punctuation after thm head
  {.5em}%     Space after thm head: " " = normal interword space;
  {}%         Thm head spec (can be left empty, meaning `normal')
\newtheoremstyle{remstyle}% name
  {6pt}%      Space above
  {6pt}%      Space below
  {\rm}%         Body font
  {}%         Indent amount (empty = no indent,\parindent = para indent)
  {\bf}% Thm head font
  {}%        Punctuation after thm head
  {.5em}%     Space after thm head: " " = normal interword space;
  {}%         Thm head spec (can be left empty, meaning `normal')
\def\Section#1{\Sec{\large #1} \setcounter{equation}{0} \vskip -6mm\indent}
\def\Sec{\@Startsection{section}1{\z@}
                                   {-3.5ex \@plus -1ex \@minus -.2ex}%
                                   {2.3ex \@plus.2ex}%
                                   {\normalfont\large\bfseries\boldmath}}
\def\@Startsection#1#2#3#4#5#6{%
  \if@noskipsec \leavevmode \fi
 \par
  \@tempskipa #4\relax
  \@afterindenttrue
  \ifdim \@tempskipa <\z@
    \@tempskipa -\@tempskipa \@afterindentfalse
  \fi
  \if@nobreak
    \everypar{}%
  \else
    \addpenalty\@secpenalty\addvspace\@tempskipa
  \fi
  \@ifstar
    {\@ssect{#3}{#4}{#5}{#6}}%
    {\@dblarg{\@Sect{#1}{#2}{#3}{#4}{#5}{#6}}}}
\def\@Sect#1#2#3#4#5#6[#7]#8{%
  \ifnum #2>\c@secnumdepth
    \let\@svsec\@empty
  \else
    \refstepcounter{#1}%
    \protected@edef\@svsec{\@seccntformat{#1}\relax}%
  \fi
  \@tempskipa #5\relax
  \ifdim \@tempskipa>\z@
    \begingroup
      #6{%
          \@hangfrom{\hskip #3\relax\@svsec \hskip -2.5mm}%
         \interlinepenalty \@M #8\@@par}
    \endgroup
    \csname #1mark\endcsname{#7}%
    \addcontentsline{toc}{#1}{%
      \ifnum #2>\c@secnumdepth \else
        \protect\numberline{\csname the#1\endcsname}%
      \fi
      #7}%
  \else
    \def\@svsechd{%
      #6{\hskip #3\relax
      \@svsec #8}%
      \csname #1mark\endcsname{#7}%
      \addcontentsline{toc}{#1}{%
        \ifnum #2>\c@secnumdepth \else
          \protect\numberline{\csname the#1\endcsname}%
        \fi
        #7}}%
  \fi
  \@xsect{#5}}
\renewenvironment{abstract}{%
        \small
        \quotation
         \noindent {\bfseries \abstractname } }%
      {\if@twocolumn\else\endquotation\fi}

\def\Subsec{\@StartSubsection{subsection}2{\z@}%
                                     {-3.25ex\@plus -1ex \@minus -.2ex}%
                                     {1.5ex \@plus .2ex}%
                                     {\normalfont\normalsize\bfseries\boldmath}}
\def\@StartSubsection#1#2#3#4#5#6{%
  \if@noskipsec \leavevmode \fi
 \par
  \@tempskipa #4\relax
  \@afterindenttrue
  \ifdim \@tempskipa <\z@
    \@tempskipa -\@tempskipa \@afterindentfalse
  \fi
  \if@nobreak
    \everypar{}%
  \else
    \addpenalty\@secpenalty\addvspace\@tempskipa
  \fi
  \@ifstar
    {\@ssect{#3}{#4}{#5}{#6}}%
    {\@dblarg{\@SubSect{#1}{#2}{#3}{#4}{#5}{#6}}}}
\def\@SubSect#1#2#3#4#5#6[#7]#8{%
  \ifnum #2>\c@secnumdepth
    \let\@svsec\@empty
  \else
    \refstepcounter{#1}%
    \protected@edef\@svsec{\@seccntformat{#1}\relax}%
  \fi
  \@tempskipa #5\relax
  \ifdim \@tempskipa>\z@
    \begingroup
      #6{%
          \@hangfrom{\hskip #3\relax\@svsec\hskip -1.5mm}%
         \interlinepenalty \@M #8\@@par}
    \endgroup
    \csname #1mark\endcsname{#7}%
    \addcontentsline{toc}{#1}{%
      \ifnum #2>\c@secnumdepth \else
        \protect\numberline{\csname the#1\endcsname}%
      \fi
      #7}%
  \else
    \def\@svsechd{%
      #6{\hskip #3\relax
      \@svsec #8}%
      \csname #1mark\endcsname{#7}%
      \addcontentsline{toc}{#1}{%
        \ifnum #2>\c@secnumdepth \else
          \protect\numberline{\csname the#1\endcsname}%
        \fi
        #7}}%
  \fi
  \@xsect{#5}}
\def\list#1#2{\ifnum \@listdepth >5\relax \@toodeep \else \global
\advance \@listdepth\@ne \fi \rightmargin \z@ \listparindent\z@
\itemindent\z@ \csname @list\romannumeral\the\@listdepth\endcsname
\def\@itemlabel{#1}\let\makelabel\@mklab \@nmbrlistfalse #2\relax
\@trivlist\parskip 0pt\parindent\listparindent \advance \linewidth
-\rightmargin \advance\linewidth -\leftmargin \advance\@totalleftmargin
\leftmargin\parshape \@ne \@totalleftmargin \linewidth \ignorespaces}
\renewcommand{\@makecaption}[2]{\begin{center}#1. #2\end{center}}
\catcode`@=12

\theoremstyle{thmstyle}
\newtheorem{thm}{\indent Theorem}[section]
\newtheorem{lem}{\indent Lemma}[section]
\newtheorem{prop}{\indent Proposition}[section]
\newtheorem{coro}{\indent Corollary}[section]
\newtheorem{defi}{\indent Definition}[section]

\newtheorem{prob}{\indent Problem}[section]
\theoremstyle{remstyle}

\newsavebox{\mygraphic}
\def\pic{\begin{picture}(0,0) \put(-210,-1250){\usebox{\mygraphic}} \end{picture}}
\newfont{\HUGEbf}{cmbx10 scaled 3500}
\definecolor{gray}{rgb}{0.9,0.9,0.9}
\def\thebibliography#1{\section*{\bf \large References}
\list{[\arabic{enumi}]} {\settowidth \labelwidth{[#1]} \leftmargin
\labelwidth \advance \leftmargin \labelsep \usecounter{enumi}}
\def\newblock{\hskip .11em plus .33em minus .07em} \footnotesize \sloppy \clubpenalty
4000 \widowpenalty 4000 \sfcode`\.=1000 \relax}

\def\BC{\mathbb C}
\def\BN{\mathbb N}
\def\BR{\mathbb R}
\def\cA{\mathcal A}
\def\cB{\mathcal B}
\def\cD{\mathcal D}

\def\cL{\mathcal L}
\def\cO{\mathcal O}

\def\cR{\mathcal R}

\def\rd{\mathrm d}
\def\dist{\mathrm{dist}}
\def\rRe{\mathrm{Re}}

\def\rdiv{\mathrm{div}}
\def\e{\mathrm e}

\def\loc{\mathrm{loc}}
\def\ri{\mathrm i}

\def\supp{\mathrm{supp}}

\def\Ga{\Gamma}

\def\Om{\Omega}
\def\al{\alpha}
\def\be{\beta}
\def\ga{\gamma}
\def\de{\delta}

\def\ve{\varepsilon}
\def\te{\theta}

\def\ka{\kappa}
\def\la{\lambda}

\def\vp{\varphi}
\def\om{\omega}
\def\f{\frac}
\def\nb{\nabla}
\def\ov{\overline}
\def\pa{\partial}
\def\wh{\widehat}

\def\tri{\triangle}

\def\casesnum#1#2{{\savebox0{$\def\item##1##2{\addno \relax
  \vphantom{##1##2} \no}
  \def\noitem##1##2{\relax
  \vphantom{##1##2}}
  \def\useritem##1##2##3{\relax
  \vphantom{##1##2}##3}
  \relax
  \ds \begin{array}{@{}r@{}}
    #2
  \end{cases}$}
  \def\item##1##2{##1 & ##2} \relax
  \def\noitem##1##2{##1 & ##2} \relax
  \def\useritem##1##2##3{##1 & ##2} \relax
  \begin{equation}
    #1
    \left\{ \begin{array}{ll}
      #2
    \end{cases} \right.
    \tag*{\usebox0}
  \end{equation}}}

\theoremstyle{definition}

\numberwithin{equation}{section}

\title{\Large\bf\boldmath Uniqueness of Inverse Source Problems for General 
Evolution Equations}

\author{\large Yavar Kian$^\dag$\qquad Yikan Liu$^\ddag$\qquad
Masahiro Yamamoto$^\star$}

\date{\today}

\begin{document}

\maketitle

\thispagestyle{first}
\renewcommand{\thefootnote}{\fnsymbol{footnote}}

\footnotetext{\hspace*{-5mm} \begin{tabular}{@{}r@{}p{14cm}@{}} &
Manuscript last updated: \today.\\
$^\dag$ & Aix Marseille Universit\'e, Universit\'e de Toulon, CNRS, CPT,
Marseille, France.\\
& E-mail: yavar.kian@univ-amu.fr\\
$^\ddag$ & Research Center of Mathematics for Social Creativity, Research
Institute for Electronic Science, Hokkaido University, N12W7, Kita-Ward,
Sapporo 060-0812, Japan. E-mail: ykliu@es.hokudai.ac.jp\\
$^\star$ & Graduate School of Mathematical Sciences, The University of
Tokyo, 3-8-1 Komaba, Meguro-ku, Tokyo 153-8914, Japan; Honorary Member
of Academy of Romanian Scientists, Splaiul Independentei Street, No.\! 54,
050094 Bucharest, Romania; Research Center of Nonlinear Problems of
Mathematical Physics, Peoples' Friendship University of Russia (RUDN
University), 6 Miklukho-Maklaya Street, Moscow 117198, Russian Federation.
E-mail: myama@ms.u-tokyo.ac.jp
\end{tabular}}

\renewcommand{\thefootnote}{\arabic{footnote}}

\begin{abstract}
In this article, we investigate inverse source problems for a wide range of PDEs
of parabolic and hyperbolic types as well as time-fractional evolution equations
by partial interior observation. Restricting the source terms to the form of
separated variables, we establish uniqueness results for simultaneously
determining both temporal and spatial components without non-vanishing
assumptions at $t=0$, which seems novel to the best of our knowledge.
Remarkably, mostly we allow a rather flexible choice of the observation time
not necessarily starting from $t=0$, which fits into various situations in
practice. Our main approach is based on the combination of the Titchmarsh
convolution theorem with unique continuation properties and time-analyticity
of the PDEs under consideration.\vskip 4.5mm

\noindent\begin{tabular}{@{}l@{ }p{10cm}} {\bf Keywords }
& Inverse source problem, evolution equation,\\
& uniqueness, Titchmarsh convolution theorem, unique continuation
\end{tabular}

\vskip 4.5mm

\noindent{\bf AMS Subject Classifications } 35R11, 35R30, 35B60

\end{abstract}

\baselineskip 14pt

\setlength{\parindent}{1.5em}

\setcounter{section}{0}
%%%%%%%%%%%%%%%%%%%%%%%%%%%%%%%%%%%%%%%%

\Section{Introduction}\label{sec-intro}

Let $\Om\subset\BR^d$ ($d=2,3,\ldots$) be a bounded domain whose
boundary $\pa\Om$ is of $C^2$ class. For $f\in C^2(\ov\Om)$, define the
elliptic operators
\[
\cA f(\bm x):=-\rdiv(\bm a(\bm x)\nb f(\bm x))+c(\bm x)f(\bm x),\quad
\cL f(\bm x):=(\cA+\bm b(\bm x)\cdot\nb)f(\bm x),\quad\bm x\in\Om,
\]
where $\cdot$ and $\nb=(\f\pa{\pa x_1},\ldots,\f\pa{\pa x_d})$ refer to the
inner product in $\BR^d$ and the gradient in $\bm x$, respectively. Here
$\bm a=(a_{j k})_{1\le j,k\le d}\in C^1(\ov\Om\,;\BR^{d\times d})$,
$\bm b=(b_1,\ldots,b_d)\in L^\infty(\Om;\BR^d)$ and $c\in L^q(\Om)$ with
$q\in(\f d2,\infty]$ are $\bm x$-dependent matrix-, vector- and scalar-valued
functions, respectively. Further, there exists a constant $\ka>0$ such that
\[
\bm a(\bm x)\bm\xi\cdot\bm\xi\ge\ka|\bm\xi|^2,\ \forall\,\bm x\in\ov\Om\,,
\ \forall\,\bm\xi\in\BR^d,\quad c(\bm x)\ge\ka\ \mbox{a.e. }\bm x\in\Om,
\]
where $|\bm\xi|^2=\bm\xi\cdot\bm\xi$. Meanwhile, for $\be\in(0,2]$ we
denote the $\be$-th order Caputo derivative in $t$ by
\[
\pa_t^\be f(t):=\left\{\!\begin{alignedat}{2}
& \f1{\Ga(\lceil\be\rceil-\be)}\int_0^t
\f{f^{(\lceil\be\rceil)}(\tau)}{(t-\tau)^{\be-\lfloor\be\rfloor}}\,\rd\tau,
& \quad & \be\in(0,1)\cup(1,2),\\
& f^{(\be)}(t), & \quad & \be=1,2,
\end{alignedat}
\right.\quad f\in C^2([0,\infty)),
\]
where $\Ga(\,\cdot\,)$, $\lceil\,\cdot\,\rceil$ and $\lfloor\,\cdot\,\rfloor$ stand
for the Gamma function, the ceiling and the floor functions, respectively.
In this paper, the order of the Caputo derivative is either a
constant in $(0,2]$ or an $x$-dependent piecewise constant function taking
value in $(0,1)$, which will be specified later in Section \ref{sec-main}. We
restrict the discussion of a variable order to $(0,1)$ because there seems very
few works on its generalization e.g.\! to $(0,2)$.

Set $\BR_+:=(0,\infty)$. In this paper, we are concerned with the following
initial-boundary value problem for a (time-fractional) evolution equation
\begin{equation}\label{eq1}
\begin{cases}
\left(\rho(\bm x)\pa_t^{\al(\bm x)}+\cL\right)u(t,\bm x)=\mu(t)h(\bm x),
& (t,\bm x)\in\BR_+\times\Om,\\
u(0,\bm x)=\pa_t^{\lceil\al\rceil-1}u(0,\bm x)=0, & \bm x\in\Om,\\
\cR u(t,\bm x)=0, & (t,\bm x)\in\BR_+\times\pa\Om,
\end{cases}
\end{equation}
where we make the following global assumptions:
\begin{gather}
\mu\in L^1(\BR_+):\mbox{compactly supported},
\quad\exists\,T_0>0\mbox{ s.t. }\mu\not\equiv0\mbox{ in }(0,T_0),\label{t1b}\\
h\in L^2(\Om),\quad\rho\in L^\infty(\Om),\quad
\exists\,\underline\rho\,,\ov\rho>0\mbox{ s.t. }
\underline\rho\le\rho\le\ov\rho\mbox{ a.e.\! in }\Om.\label{eq-rho}
\end{gather}
Here $\cR$ denotes either the Dirichlet boundary condition or the Neumann
boundary condition associated with the principal part $\bm a$ of the elliptic
operator $\cA$, i.e.,
\[
\cR f(\bm x)=f(\bm x)\quad\mbox{or}\quad
\cR f(\bm x)=\bm a(\bm x)\nb f(\bm x)\cdot\bm\nu(\bm x),\quad
\bm x\in\pa\Om,
\]
where $\bm\nu(\bm x)$ is the outward unit normal vector to $\pa\Om$ at
$\bm x\in\pa\Om$. The solution to \eqref{eq1} is understood in
the weak sense, whose precise definition will be given in Definition \ref{d1}
(see Section \ref{sec-main}).

This paper focuses on the uniqueness issue of the following inverse source
problem concerning \eqref{eq1}.

\begin{prob}\label{ISP}
Let $u$ satisfy $\eqref{eq1}$ with \eqref{t1b}--\eqref{eq-rho},
$I\subset\BR_+$ be a finite open interval and $\om\subset\Om$ a nonempty
open subset. Under certain assumptions$,$ determine $h$ with given $\mu$
or determine $h,\mu$ simultaneously by the partial interior observation of $u$
in $I\times\om$.
\end{prob}
%%%%%%%%%%%%%%%%%%%%

%1.4 Motivations
The governing equation in \eqref{eq1} takes the form of a rather general
(time-fractional) evolution equation which includes the traditional parabolic
and hyperbolic ones. Correspondingly, Problem \ref{ISP} also covers a wide
range of inverse source problems arising in several scientific areas including
medical imaging, optical tomography, seismology and environmental problems.
For a constant $\al\in(0,2)$, Problem \ref{ISP} corresponds to the recovery of a
source in typical or anomalous diffusion processes appearing in geophysics,
biology and environmental science (see e.g. \cite{JR,NSY}). For $\al=2$, our
inverse problem can be associated with the determination of an acoustic
source with applications in area such as medical imaging and seismology. For
instance, Problem \ref{ISP} with $\al=2$ can be applied to some inverse
diffraction and near-field holography problems (see e.g.
\cite[Chapter 2.2.5]{GC}). We mention also that problem \eqref{eq1} with a
variable order $\al(\bm x)$ is considered as a model for diffusion phenomenon
in some complex media, where the variation of $\al$ is due to the presence of
heterogeneous regions. In that context, Problem \ref{ISP} can be seen as the
determination of a source appearing in a diffusion process associated with
problems in chemistry \cite{CZZ}, biology \cite{GN} and physics \cite{StS, ZLL}.
%%%%%%%%%%%%%%%%%%%%

%1.5 A short bibliography of inverse source problems
Among the various formulations of inverse problems, inverse source problems
have gathered consistent popularity owing to their theoretical and practical
significance. The interested reader can refer to \cite{I,LLY2} for an overview of
these problems. For $\al=1,2$, these problems have been studied extensively
in the last decades. Without being exhaustive, we refer to
\cite{cho,IY1,IY2,JLY,KSS,Ya95,yam,Ya1999}, among which the approach of
\cite{IY1,IY2,Ya1999} was based on the Bukhgeim-Klibanov method introduced
in \cite{BK}. For determining the temporal components in source terms of
time-fractional diffusion equations and hyperbolic systems, we mention
\cite{FK,LRY,LZ,SY} and \cite{BHKY,HK,HKLZ}, respectively. In the same spirit,
\cite{Ik,KW} were devoted to the determination of information about the
support of general source terms in parabolic equations, and \cite{AE,EH} dealt
with that of time-dependent point sources. It was proved in \cite{KY2} that
some classes of time-dependent source terms appearing in similar problems to
\eqref{eq1} with a constant $\al\in(0,2)$ can be reconstructed from boundary
measurements when $\Om$ is a cylindrical domain. The approach of
\cite{KY2} has been recently extended by \cite{JKZ} to similar equations with
time-dependent elliptic operators. Meanwhile, in \cite{JK} the authors proved
the recovery of general source terms from the full knowledge of the solution in
a time interval $(t_0,T)$ with $t_0\in(0,T)$. 

Let us emphasize that almost all above mentioned results assumed the
unknown functions to be independent of either the time variable or one space
variable. For the first class of source terms, the known component in the
source term may depend on time. Among all these results, the strategy
allowing the recovery of such classes of source terms requires the
non-vanishing assumption of the known component of the source term at
$t=0$. Recently, \cite{KSXY} established one of the first results of recovering
the source term in \eqref{eq1} without assuming $\mu(0)\ne0$ or more
generally, only requiring $\supp\,\mu\subset[0,T)$. Moreover, \cite{KSXY} also
proved the simultaneous determination of $\mu,h$ in \eqref{eq1} provided that
the restriction of $\mu$ to a subinterval of $\BR_+$ is known and admits an
analytic extension. On the same direction of \cite{KSXY}, in this paper we will
demonstrate the possibility of determining $\mu,h$ simultaneously in more
general settings. By the way, we restrict the source term to the form
$\mu(t)h(\bm x)$ of complete separated variables in view of the obstruction for
this problem described e.g.\! in \cite[\S1.3.1]{KSXY}.

Meanwhile, another highlight of Problem \ref{ISP} is the relaxation of the
observation time. In most literature on inverse source problems, the
observation was assumed to start from $t=0$. However in practice, usually the
data is only available after the occurrence of some unpredictable accidents.
Therefore, it is reasonable to generalize the observation time to a finite interval
$I\subset\BR_+$ not necessarily starting from $t=0$. It turns out that such a
relaxation gives affirmative answers to Problem \ref{ISP} in most situations with
unfortunate exceptions of $\al=1,2$.

The remainder of this article is organized as follows. In Section \ref{sec-main},
we collect the preliminaries to deal with Problem \ref{ISP} and state three main
results on uniqueness according to the choices of $\al$ along with comments
about these results. Then the next three sections are devoted to the proofs for
different cases of $\al$ respectively. Finally, some technical
details will be provided in Section \ref{sec-app}.
%%%%%%%%%%%%%%%%%%%%%%%%%%%%%%%%%%%%%%%%

\Section{Preliminaries and Main Results}\label{sec-main}

To begin with, we first fix the frequently used notations in the sequel.
Throughout this article, $\BN:=\{1,2,\ldots\}$ stands for positive integers. By
$H^s(\Om)$ ($s\in\BR$), $H_0^1(\Om)$, $W^{1,1}(0,T)$, etc.\! we denote the
usual Sobolev spaces (see Adams \cite{A75}). The weighted $L^2$-space in
$\Om$ with the weight $\rho$ is denoted by $L^2(\Om;\rho\,\rd\bm x)$.
Given a nonempty subset $\om\subset\Om$, the duality pairing between
$C_0^\infty(\om)$ and the space $\cD'(\om)$ of Schwartz distributions is
denoted by ${}_{\cD'(\om)}\langle\,\cdot\,,\,\cdot\,\rangle_{C^\infty_0(\om)}$.
For Banach spaces $X$ and $Y$, we denote the collection of bounded linear
operators from $X$ to $Y$ by $\cB(X,Y)$, which is abbreviated as $\cB(X)$
when $X=Y$. For $f\in L^1(\BR_+)$, its Laplace transform is denoted by
\[
\wh f(p):=\int_0^\infty\e^{-p t}f(t)\,\rd t.
\]
Further notations will be introduced whenever needed.

Before stating the main results, we briefly revisit the definition and the
well-posedness of the governing system \eqref{eq1}.

\begin{defi}\label{d1}
Let the coefficients and the source term in \eqref{eq1} satisfy
\eqref{t1b}--\eqref{eq-rho}. We say that
$u\in L_\loc^1(\BR_+;L^2(\Om))$ is a weak solution to \eqref{eq1} if it satisfies
the following conditions.
\begin{enumerate}
\item[{\rm1.}]
$p_0:=\inf\{\ve>0\mid\e^{-\ve t}u\in L^1(\BR_+;L^2(\Om))\}<\infty$.
\item[{\rm2.}] For all $p>p_0,$ the Laplace transform $\wh u(p;\,\cdot\,)$ of
$u(t,\,\cdot\,)$ with respect to $t$ solves the following boundary value
problem 
\begin{equation}\label{d1a}
\begin{cases}
(\cL+p^\al\rho)\wh u(p;\,\cdot\,)=\wh\mu(p)h & \mbox{in }\Om,\\
\cR\wh u(p;\,\cdot\,)=0 & \mbox{on }\pa\Om.
\end{cases}
\end{equation}
\end{enumerate}
\end{defi}

For a constant $\al\in(0,2)$, the above definition of weak
solutions is equivalent to that of mild solutions (see \cite{KSY,KY,KY2}) with
smooth coefficients. For $\bm b\equiv\bm0$, we refer to
\cite[Theorem 2.3]{KY1} (see also \cite{KY}) for the unique existence of a weak
solution $u\in L^1_\loc(\BR_+;L^2(\Om))$ to \eqref{eq1}. For $\al\in(0,1]$ and
$\bm b\not\equiv\bm0$, we refer to \cite[pp.13--15]{Kia} for the unique
existence of a weak solution $u\in L^1_\loc(\BR_+;H^1(\Om))$ to \eqref{eq1}.
For $\al\in L^\infty(\Om)$ and $\bm b\equiv\bm0$, in
Proposition \ref{pp2} we prove the unique existence of a weak solution
$u\in L^1_\loc(\BR_+;L^2(\Om))$ to \eqref{eq1}. We point out that for
$\al=1,2$, the weak solutions to \eqref{eq1} defined above coincide with the
classical variational solution to the corresponding parabolic and hyperbolic
equations, respectively.

Now we are well prepared to state our main results regarding Problem
\ref{ISP}. Due to the essential difference of the problems and the
corresponding results, we divide the statement into three contexts according
to difference choices of $\al$, namely
\begin{enumerate}
\item $\al\equiv2$,
\item $\al\in(0,2)$ is a constant, and
\item $\al:\Om\longrightarrow(0,1)$ is a piecewise constant.
\end{enumerate}
In each case, we first establish a uniqueness principle for \eqref{eq1} and then,
as a direct consequence, state a uniqueness result regarding Problem \ref{ISP}
\bigskip
%%%%%%%%%%%%%%%%%%%%

{\bf Case 1 } $\al\equiv2$. We denote by $\dist$ the Riemannian distance of
$\ov\Om$ equipped with the metric $g$, where $g=g(\bm x)$ is the inverse of
the matrix $(\rho^{-1}(\bm x)a_{j k}(\bm x))_{1\le j,k\le d}$.

\begin{thm}\label{t1} 
Let $u$ satisfy \eqref{eq1} with $\al\equiv2,$ where we assume
\eqref{t1b}--\eqref{eq-rho} and additionally $\bm b\equiv\bm0,$
$c\in L^\infty(\Om),$ $\rho\in C^1(\ov\Om)$. For any nonempty open subset
$\om\subset\Om,$ if
\begin{equation}\label{t1a}
T\ge T_0+\sup_{\bm x\in\Om}\dist(\bm x,\om),
\quad\mbox{where }\dist(\bm x,\om):=\inf_{\bm y\in\om}\dist(\bm x,\bm y),
\end{equation}
then $u=0$ in $(0,T)\times\om$ implies $h\equiv0$.
\end{thm}

An immediate application of Theorem \ref{t1} to Problem \ref{ISP} indicates the
unique determination of the spatial component $h$ in the source term when
the temporal component $\mu$ is known.

\begin{coro}\label{c1} 
Let the conditions in Theorem $\ref{t1}$ be fulfilled and $u_i$ satisfy
\eqref{eq1} with $\al\equiv2$ and $h_i\in L^2(\Om)$ $(i=1,2),$ where
$\mu\in L^1(\BR_+)$ is known. Then $u_1=u_2$ in
$(0,T)\times\om$ implies $h_1= h_2$ in $\Om$.
\end{coro}

To the best of our knowledge, the uniqueness principle of
Theorem \ref{t1} and the associated uniqueness in Corollary \ref{c1} are the
first results on determining the spatial component $h$ of the source term in a
hyperbolic equation without assuming $\mu(0)\ne0$ or $T=\infty$. Indeed, it
seems that all other similar results required either some non-vanishing
conditions at $t=0$ (see e.g. \cite{IY2,JLY,Ya95,Ya1999}) or infinite
observation time (e.g. \cite{HK}). Especially, the non-vanishing condition in our
context reads $\mu(0)\ne0$, which restricts the problem under consideration
to such a situation that the phenomenon of interest should start before
observation. By removing this condition, we make the results of Theorem
\ref{t1} and Corollary \ref{c1} more flexible to allow the measurement to start
before the appearance of some unknown phenomenon.\bigskip
%%%%%%%%%%%%%%%%%%%%

{\bf Case 2 } $\al\in(0,2)$ is a constant.

\begin{thm}\label{t2}
Let $u$ satisfy \eqref{eq1} with a constant $\al\in(0,2),$ where we assume
\eqref{t1b}--\eqref{eq-rho} and $h=0$ in $\om$. Moreover$,$
we restrict $\rho\equiv1$ if $\al\in(0,1]$ and $\bm b\equiv\bm0$ if
$\al\in(1,2)$. Then for any nonempty open subset $\om\subset\Om,$
\begin{enumerate}
\item[{\rm(1)}] If $\al\ne1,$ then for any $T\ge T_0$ and any $T_1\in[0,T),$
the following implication holds true.
\begin{equation}\label{tt2aa}
(u=0\mbox{ in }(T_1,T)\times\om)
\quad\Longrightarrow\quad(h\equiv0\mbox{ in }\Om).
\end{equation}
\item[{\rm(2)}] If $\al=1,$ then the implication \eqref{tt2aa} holds true provided
that $T_1=0$.
 \end{enumerate}
 \end{thm}

Applying Theorem \ref{t2}, we can prove the uniqueness for determining
$\mu$ and $h$ simultaneously in the source term of \eqref{eq1}, provided 
that $\mu$ satisfies \eqref{t1b} and is known in $(0,T_0)$.

\begin{coro}\label{c2}
Let the conditions in Theorem $\ref{t2}$ be fulfilled and $u_i$ satisfy
\eqref{eq1} with a constant $\al\in(0,2),$ $\mu_i\in L^1(\BR_+)$
being compactly supported and $h_i\in L^2(\Om)\ (i=1,2),$ where we assume
\begin{gather}
\mu_1\mbox{ satisfies \eqref{t1b}},\quad\mu_1=\mu_2\mbox{ in }(0,T_0),
\label{c2a}\\
h_1=h_2\mbox{ in }\om,\quad h_1\not\equiv0\mbox{ in }\Om.\label{c2aa}
\end{gather}
Then
\begin{enumerate}
\item[{\rm(1)}] If $\al\ne1,$ then for any $T\ge T_0$ and any $T_1\in[0,T_0),$
the following implication holds true.
\begin{equation}\label{c2b}
(u_1=u_2\mbox{ in }(T_1,T)\times\om)
\quad\Longrightarrow\quad(\mu_1\equiv\mu_2\mbox{ in }(0,T)
\mbox{ and }h_1\equiv h_2\mbox{ in }\Om).
\end{equation}
\item[{\rm(2)}] If $\al=1,$ then the implication \eqref{c2b} holds true provided
that $T_1=0$.
\end{enumerate}
\end{coro}

In \cite{JLLY,KSXY}, similar results to Theorem \ref{t2} were
established by requiring $T_1=0$. In this sense, Theorem \ref{t2} greatly
improves the flexibility in the choice of the observation time. On the other
hand, Theorem \ref{t2} extends the uniqueness principle in
\cite[Theorem 1.1]{KSXY} by removing the requirement
$\supp\,\mu\subset[0,T)$. Actually, the uniqueness principle of Theorem
\ref{t2} even holds true with measurement taken in $(0,T_0)$ where $\mu$ is
not uniformly vanishing (see \eqref{t1b}). As a direct consequence of Theorem
\ref{t2}, Corollary \ref{c2} claims the uniqueness of the simultaneous
determination of both temporal and spatial components of the source term
under the tolerable extra assumptions \eqref{c2a}--\eqref{c2aa}. To the best of
our knowledge, Corollary \ref{c2} is the first result on completely determining a
source term of separated variables stated in such a general context. The only
comparable result seems to be \cite[Theorem 1.3]{KSXY}, which additionally
requires that the restriction of $\mu$ to $(0,T_0)$ admits a holomorphic
extension to some neighborhood of $\BR_+$. In that sense, Corollary \ref{c2}
generalizes \cite[Theorem 1.3]{KSXY} considerably. In addition, for
$\al\in(0,1]$ we also allow the presence of a convection term $\bm b$ in
\eqref{eq1}, which breaks the symmetry of the elliptic part and thus any
solution representation of \eqref{eq1} using the eigensystem. Therefore, instead
of the method in \cite{KSXY}, we adopt the idea in \cite{JLLY} to combine the
time-analyticity of the solution with the unique continuation property of
parabolic equations.\bigskip
%%%%%%%%%%%%%%%%%%%%

{\bf Case 3 } $\al:\Om\longrightarrow(0,1)$ is a piecewise constant. More
precisely, we assume that for a fixed $N\in\BN$, there exist constants
$\al_\ell\in(0,1)$ and open subdomains $\Om_\ell\subset\Om$
($\ell=1,2,\ldots,N$) with Lipschitz boundaries such that
\begin{gather}
\ov\Om=\bigcup_{\ell=1}^N\ov{\Om_\ell}\,,
\quad\Om_\ell\cap\Om_m=\emptyset\ (\ell\ne m),\label{bb}\\
\al(\bm x)=\al_\ell\ (\bm x\in\Om_\ell,\ \ell=1,\ldots,N),
\quad 0<\al_1<\al_2<\cdots<\al_N<\min\{2\al_1,1\}.\label{vo}
\end{gather}

\begin{thm}\label{t3}
Let $u$ satisfy \eqref{eq1} with $\al\in L^\infty(\Om)$ satisfying
{\rm\eqref{bb}--\eqref{vo},} where we assume
\eqref{t1b}--\eqref{eq-rho}, $\bm b\equiv\bm0$ and there exists
an open subdomain $\cO\subset\Om$ such that
\begin{equation}\label{t3b}
\left(\bigcup_{\ell=1}^N\pa\Om_\ell\right)\setminus\pa\Om\subset\cO,
\quad h=0\mbox{ in }\cO.
\end{equation}
Then for any open subset $\om\subset\Om$ satisfying
\begin{equation}\label{t3aa}
h=0\mbox{ in }\om,\quad\om\cap\cO\ne\emptyset,
\end{equation}
any $T\ge T_0$ and any $T_1\in[0,T),$ the implication \eqref{tt2aa} holds
true.
\end{thm}

For a better understanding of the readers, in Figure \ref{fig-geo} we illustrate a
typical situation of the geometrical assumptions in Theorem \ref{t3}.
\begin{figure}[htbp]\centering
\input{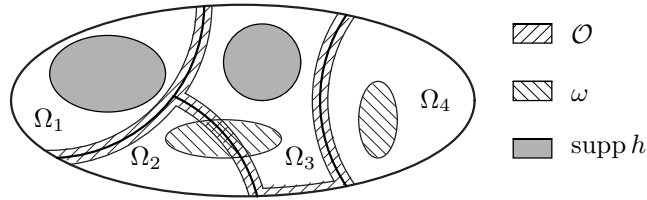}
\caption{A typical situation of the geometrical assumptions in Theorem \ref{t3}
with $N=4$. The subdomain $\cO$ is a connect open neighborhood of the
interior boundaries of $\Om_\ell$, which $\om$ should intersect. The support
of $h$ and $\cO\cup\om$ are disjoint.}\label{fig-geo}
\end{figure}

Parallel to that of Corollary \ref{c2}, we can apply Theorem \ref{t3} to prove
the unique determination of $\mu,h$ simultaneously in the source term of
\eqref{eq1} under the same additional condition on $\mu$.

\begin{coro}\label{c3} 
Let the conditions in Theorem $\ref{t3}$ be fulfilled and $u_i$ satisfy
\eqref{eq1} with $\al\in L^\infty(\Om)$ satisfying {\rm\eqref{bb}--\eqref{vo},}
$\mu_i\in L^1(\BR_+)$ being compactly supported and
$h_i\in L^2(\Om)\ (i=1,2),$ where we assume \eqref{c2a} and
\begin{equation}\label{eq-assume-h}
h_1=h_2\mbox{ in }\om\cup\cO,\quad h_1\not\equiv0\mbox{ in }\Om.
\end{equation}
Then for any open subdomain $\om\subset\Om$ satisfying $\eqref{t3aa},$ any
$T\ge T_0$ and any $T_1\in[0,T_0),$ the implication \eqref{c2b} holds true.
\end{coro}
%%%%%%%%%%%%%%%%%%%%

%In general, the above results improve the previous uniqueness results of
%related inverse source problems for various evolution equations in several
%aspects. In the remaining part of this section, we describe the novelty of each
%of these results in comparison with those in existing literature.

To the best of our knowledge, Theorem \ref{t3} and Corollary
\ref{c3} are the first results on the uniqueness of inverse source problems for
time-fractional diffusion equations with variable orders. We restrict the order
to a piecewise constant function due to its significance from the practical point
of view. In that context, we obtain some uniqueness results comparable to the
case of a constant order, provided that the support of spatial component in the
source term does not pass through the regions where $\al$ takes different
values (see condition \eqref{t3b}).

As was mentioned before, the highlight in Theorems \ref{t2}--\ref{t3} and
Corollaries \ref{c2}--\ref{c3} is the relaxation of the choice of the observation
time. Indeed, for $\al\ne1$, all other related results required that the
observation data should be taken from $t=0$. In contrast, Theorems
\ref{t2}--\ref{t3} and Corollaries \ref{c2}--\ref{c3} only assume that the
measurement does not terminate earlier than $T_0$ in \eqref{t1b}. Recalling
that \eqref{t1b} means $\mu\not\equiv0$ in $(0,T_0)$, such an assumption is
the minimal necessary that one can expect. This remarkable relaxation not
only allows the absence of data near the initial time, but also enables the
delayed measurement carried out even after the disappearance of the source
under consideration. Note that such restriction of the measurement requires
the application of a result in \cite{JK}, which cannot be applied to $\al=1$. It is
not clear whether the implication \eqref{tt2aa} still holds true for $\al=1$ with
$T_1>0$.
%%%%%%%%%%%%%%%%%%%%%%%%%%%%%%%%%%%%%%%%

\Section{Proof of Theorem \ref{t1} and Corollary \ref{c1}}

\begin{proof}[Proof of Theorem $\ref{t1}$]
Let $u$ be the solution of \eqref{eq1} with $\al=2$ and assume that the
condition
\begin{equation}\label{t1c}
u=0\quad\mbox{in }(0,T)\times\om
\end{equation}
is fulfilled, where $T$ satisfies \eqref{t1a}. We will show that \eqref{t1c} implies
$h\equiv0$.

According to the Duhamel's principle, it is readily seen that $u$ takes the form
of
\begin{equation}\label{t1d}
u(t,\,\cdot\,)=\int_0^t\mu(t-\tau)\,v(\tau,\,\cdot\,)\,\rd\tau,
\end{equation}
where $v$ solves
\[
\begin{cases}
(\rho\,\pa_t^2+\cA)v=0 & \mbox{in }(0,T)\times\Om,\\
v=0,\ \pa_t v=h & \mbox{in }\{0\}\times\Om,\\
\cR v=0 & \mbox{on }(0,T)\times\pa\Om.
\end{cases}
\]
By Lions and Magenes \cite{LM1}, we know
$v\in C([0,T];H^1(\Om))\cap C^1([0,T];L^2(\Om))$. Then for arbitrarily fixed
$\psi\in C^\infty_0(\om)$, it follows from \eqref{t1d} and \eqref{t1c} that
\[
0={}_{\cD'(\om)}\langle u(t,\,\cdot\,),\psi\rangle_{C^\infty_0(\om)}
=\int_0^t\mu(t-\tau)\,v_\psi(\tau)\,\rd\tau,\quad0<t<T,
\]
where
$v_\psi(t):={}_{\cD'(\om)}\langle v(t,\,\cdot\,),\psi\rangle_{C^\infty_0(\om)}$.
Since $\mu\in L^1(0,T)$, we apply the Titchmarsh convolution theorem (see
\cite[Theorem VII]{Ti}) to deduce that there exist $\tau_1,\tau_2\in[0,T]$
satisfying $\tau_1+\tau_2\ge T$ such that
\[
\mu\equiv0\mbox{ in }(0,\tau_1),\quad v_\psi\equiv0\mbox{ in }(0,\tau_2).
\]
In view of the key assumption \eqref{t1b}, we have $\tau_1\le T_0$. Then it
follows from \eqref{t1a} that 
\[
\tau_2\ge T-\tau_1\ge T-T_0\ge\sup_{\bm x\in\Om}\dist(\bm x,\om)=:T_1,
\]
which implies
\[
v_\psi(t)={}_{\cD'(\om)}\langle v(t,\,\cdot\,),\psi\rangle_{C^\infty_0(\om)}
=0,\quad t\in[0,T_1].
\]
Since in this identity, $\psi\in C^\infty_0(\om)$ was chosen arbitrary and
$T_1$ is independent of $\psi$, we deduce that
\begin{equation}\label{t1f}
v=0\quad\mbox{in }(0,T_1)\times\om.
\end{equation}

On the other hand, recalling the metric $g$ defined on $\ov\Om\,$, we
introduce the Laplace-Beltrami operator $\tri_g$ associated with the
Riemannian manifold $(\ov\Om\,,g)$. Then there exists
$\bm K\in C(\ov\Om\,;\BR^d)$ such that
\[
(\pa_t^2-\tri_g+\bm K\cdot\nb+\rho^{-1}c)v
=(\pa_t^2+\rho^{-1}\cA)v=0\quad\mbox{in }(0,T)\times\Om.
\]
Now we define the odd extension of $v$ in $(-T,T)\times\Om$, still denoted by
$v$, by
\[
v(t,\,\cdot\,)=-v(-t,\,\cdot\,)\quad\mbox{in }\Om,\ -T<t<0.
\]
Using the fact that $v(0,\,\cdot\,)=0$, we see that $v\in H^1((-T,T)\times\Om)$.
Meanwhile, in view of \eqref{t1f}, we see that $v$ satisfies
\[
\begin{cases} 
(\pa_t^2-\tri_g+\bm K\cdot\nb+\rho^{-1}c)v=0
& \mbox{in }(-T_1,T_1)\times\Om,\\
v=0 & \mbox{in }(-T_1,T_1)\times\om.
\end{cases}
\]
Therefore, applying a global unique continuation result theorem similar to
\cite[Theorem 3.16]{KKL} (see also \cite[Theorem A.1]{KMO}) derived from the
local unique continuation result of \cite{RZ,T1}, we obtain
\[
h=\pa_t v(0,\,\cdot\,)=0\quad
\mbox{in }\{\bm x\in\Om\mid\mathrm{dist}(\bm x,\om)\le T_1\}.
\]
Then it follows from the definition of $T_1$ that
$\{\bm x\in\Om\mid\dist(\bm x,\om)\le T_1\}=\Om$ or equivalently
$h\equiv0$ in $\Om$. This completes the proof of Theorem \ref{t1}.
\end{proof}
%%%%%%%%%%%%%%%%%%%%

\begin{proof}[Proof of Corollary $\ref{c1}$]
Introducing $u:=u_1-u_2$, it is readily seen that $u$ satisfies
\[
\begin{cases}
(\rho(\bm x)\pa_t^2u+\cA)u(t,\bm x)=\mu(t)(h_1-h_2)(\bm x),
& (t,\bm x)\in\BR_+\times\Om,\\
u(0,\bm x)=\pa_t u(0,\bm x)=0, & \bm x\in\Om,\\
\cR u(t,\bm x)=0, & (t,\bm x)\in\BR_+\times\pa\Om.
\end{cases}
\]
Meanwhile, the condition $u_1=u_2$ in $(0,T)\times\om$ implies $u=0$ in
$(0,T)\times\om$. Since $T$ satisfies \eqref{t1b}, Theorem \ref{t1} immediately
implies $h_1-h_2\equiv0$ or equivalently $h_1\equiv h_2$ in $\Om$.
\end{proof}
%%%%%%%%%%%%%%%%%%%%%%%%%%%%%%%%%%%%%%%%

\Section{Proof of Theorem \ref{t2} and Corollary \ref{c2}}

In this section, we investigate the case of a constant $\al\in(0,2)$ in
\eqref{eq1}. To this end, a similar solution representation to \eqref{t1d} is
necessary, which requires some preparations.

In the case of $\al\in(0,1]$, we specify $\cL$ as the unbounded elliptic
operator acting on $L^2(\Om)$ with the domain
$\cD(\cL)=\{f\in H^2(\Om)\mid\cR f=0\mbox{ on }\pa\Om\}$. From now on, for
all $s,\te\in[0,\infty)$, we denote by $D_{s,\te}$ the set
\[
D_{s,\te}:=\{s+r\,\e^{\ri\be}\mid r>0,\ \be\in(-\te,\te)\}.
\]
According to \cite[Theorem 2.1]{A} (see also \cite[Theorem 2.5.1]{LLMP}),
there exists $\te_0\in(\f\pi2,\pi)$ and $s_0\ge0$ such that $D_{s_0,\te_0}$ is in
the resolvent set of $\cL$. Moreover, there exists a constant $C>0$ depending
on $\cL$ and $\Om$ such that
\[%begin{equation}\label{ess}
\|(\cL+z)^{-1}\|_{\cB(L^2(\Om;\rho\,\rd\bm x))}
+|z|^{-1}\|(\cL+z)^{-1}\|_{\cB(L^2(\Om;\rho\,\rd\bm x);H^2(\Om))}
\le C|z|^{-1},\quad z\in D_{s_0,\te_0}.
\]%end{equation}
Here we employed the fact that $L^2(\Om)=L^2(\Om;\rho\,\rd\bm x)$ in the
sense of the norm equivalence thanks to \eqref{eq-rho}. We fix
$\te_1\in(\f\pi2,\te_0)$, $\de\in\BR_+$ and consider a contour $\ga(\de,\te_1)$
in $\BC$ defined by
\[%begin{equation}\label{eq-def-ga}
\ga(\de,\te_1):=\ga_-(\de,\te_1)\cup\ga_0(\de,\te_1)\cup\ga_+(\de,\te_1)
\]%end{equation}
oriented in the counterclockwise direction, where
\[
\ga_0(\de,\te_1):=\{s_0+\de\,\e^{\ri\be}\mid\be\in[-\te_1,\te_1]\},
\quad\ga_\pm(\de,\te_1):=\{s_0+r\,\e^{\pm\ri\te_1}\mid r\ge\de\}
\]
with double signs in the same order.

Let $\te_2\in(0,\te_1-\f\pi2)$. Applying the above properties of $\cL$, for
$\al\in(0,1]$ and $z\in D_{0,\te_2}$, we can define an operator
$S_1(z)\in\cB(L^2(\Om))$ by 
\[
S_1(z)u_0=\f1{2\pi\,\ri}\int_{\ga(\de,\te_1)}\e^{z p}(\cL+p^\al)^{-1}u_0\,\rd p,
\quad u_0\in L^2(\Om).
\]
We recall the following property of the map $z\longmapsto S_1(z)$.

\begin{lem}\label{l1}
{\rm(see \cite[Lemma 2.4]{Kia}) } For all $\be\in[0,1],$ the map
$z\longmapsto S_1(z)$ is analytic from $D_{0,\te_2}$ to
$\cB(L^2(\Om); H^{2\be}(\Om))$. Moreover$,$ there exists a constant $C>0$
depending only on $\cL$ and $\Om$ such that
\begin{equation}\label{l1a}
\|S_1(z)\|_{\cB(L^2(\Om;\rho\,\rd\bm x);H^{2\be}(\Om))}
\le C|z|^{\al(1-\be)-1}\e^{s_0\,\rRe\,z},\quad z\in D_{0,\te_0}.
\end{equation}
\end{lem}

In addition to this property, we combine \eqref{l1a} with the arguments for
\cite[Theorem 1.1]{KSY} and \cite[Remark 1]{KSY} to deduce that for
$\mu\in L^\infty(\BR_+)$, \eqref{eq1} admits a unique weak solution
$u\in C([0,\infty);L^2(\Om))$ taking the form
\begin{equation}\label{t2cd}
u(t,\,\cdot\,)=\int_0^t\mu(t-\tau)\,S_1(\tau)h\,\rd\tau.
\end{equation}
Using some arguments similar to that for
\cite[Proposition 6.1]{KSXY}, one can show that the identity
\eqref{t2cd} also holds true for $\mu\in L^1(\BR_+)$ with a
compact support.

Likewise, in the case of $\al\in(1,2)$ with $\bm b\equiv\bm0$, we turn to the
eigensystem $\{(\la_n,\vp_n)\}_{n\in\BN}$ of the self-adjoint operator
$A=\rho^{-1}\cA$ acting on $L^2(\Om;\rho\,\rd\bm x)$, with the boundary
condition $\cR u=0$ in $\pa\Om$, to define an operator $S_2(t)$ for $t>0$ by
\[
S_2(t)u_0=t^{\al-1}\sum_{n=1}^\infty
E_{\al,\al}(-t^\al\la_n)\,(u_0,\vp_n)_\rho\,\vp_n,\quad u_0\in L^2(\Om),
\]
where $(\,\cdot\,,\,\cdot\,)_\rho$ denotes the inner product of
$L^2(\Om;\rho\,\rd\bm x)$. Here $E_{\al,\al}(\,\cdot\,)$ is the Mittag-Leffler
function defined by 
\[
E_{\al,\al}(z)=\sum_{k=0}^\infty\f{z^k}{\Ga(\al k+\al)},\quad z\in\BC.
\]
Similarly as before, applying \cite[Theorem 1.6]{P} and following the arguments
used in \cite[Theorem 1.1]{KY}, one can check that $S_2$ is well-defined as an
element of $L^1_\loc(\BR_+;\cB(L^2(\Om)))$. Moreover, we can show that
\eqref{eq1} admits a unique weak solution
$u\in C(\BR_+;L^2(\Om))$ taking the form
\begin{equation}\label{t2dd}
u(t,\,\cdot\,)=\int_0^t\mu(t-\tau)\,S_2(\tau)h\,\rd\tau.
\end{equation}

For the proof of Theorem \ref{t2}, we also invoke the
Riemann-Liouville integral operator $I^\be$ and the Riemann-Liouville
derivative $D_t^\be$ for $\be\in(0,2]$:
\[
I^\be f(t):=\left\{\!\begin{alignedat}{2}
& f(t), & \ & \be=0,\\
& \f1{\Ga(\be)}\int_0^t\f{f(\tau)}{(t-\tau)^{1-\be}}\,\rd\tau, & \ & \be>0
\end{alignedat}\right.\ (f\in C([0,\infty))),\quad D_t^\be:=
\f{\rd^{\lceil\be\rceil}}{\rd t^{\lceil\be\rceil}}\circ I^{\lceil\be\rceil-\be},
\]
where $\circ$ denotes the composite. We need the following technical lemma.

\begin{lem}\label{lll1}
Let $h\in L^1(0,T),$ $\be\in(0,1)\cup(1,2),$ $\la>0$ and $w\in L^1(0,T)$ be
given by
\[
w(t):=\int_0^t(t-\tau)^{\be-1}E_{\be,\be}(-\la(t-\tau)^{\be})h(\tau)\,\rd\tau.
\]
Then
\[
(D_t^\be+\la)w(t)=h(t),\quad t\in(0,T)
\]
holds true in the sense of distributions.
\end{lem}

The result of Lemma \ref{lll1} is rather classical for smooth
function $h$. However, we have not found a clear proof of
such result for $h\in L^1(0,T)$. For this reason, we provide its full proof
here.

\begin{proof}[Proof of Lemma $\ref{lll1}$]
Let $\be\in(0,1)\cup(1,2)$ and fix a sequence
$\{h_k\}_{k\in\BN}\subset C^\infty_0(0,T)$ such that
\[
\lim_{k\to\infty}\|h_k-h\|_{L^1(0,T)}=0.
\]
One can check that, for $\{w_k\}_{k\in\BN}$ defined by
\[
w_k(t):=\int_0^t(t-\tau)^{\be-1}E_{\be,\be}(-\la(t-\tau)^{\be})h_k(\tau)\,\rd\tau,
\]
there holds
\begin{equation}\label{lll1c}
(D_t^\be+\la)w_k(t)=h_k(t),\quad t\in(0,T).
\end{equation}
Moreover, applying \cite[Theorem 1.6]{P} and Young's
convolution inequality, we get
\[
\limsup_{k\to\infty}\|w_k-w\|_{L^1(0,T)}
\le C\|t^{\be-1}\|_{L^1(0,T)}\limsup_{k\to\infty}\|h_k-h\|_{L^1(0,T)}=0.
\]
In the same way, we obtain
\begin{equation}\label{lll1e}
\limsup_{k\to\infty}\|I^\be w_k-I^\be w\|_{L^1(0,T)}
\le C\|t^{\lfloor\be\rfloor-\be}\|_{L^1(0,T)}
\limsup_{k\to\infty}\|w_k-w\|_{L^1(0,T)}=0.
\end{equation}
Therefore, fixing $\chi\in C^\infty_0(0,T)$ and applying \eqref{lll1e}, we find
\begin{align}
\lim_{k\to\infty}{}_{\cD'(0,T)}\langle D_t^\be w_k,\chi\rangle_{C^\infty_0(0,T)}
& =\lim_{k\to\infty}{}_{\cD'(0,T)}\langle\pa_t^{\lceil\be\rceil}
I^{\lceil\be\rceil-\be}w_k,\chi\rangle_{C^\infty_0(0,T)}
\nonumber\\
& =(-1)^{\lceil\be\rceil}\lim_{k\to\infty}{}_{\cD'(0,T)}\langle
I^{\lceil\be\rceil-\be}
w_k,\chi^{(\lceil\be\rceil)}\rangle_{C^\infty_0(0,T)}\nonumber\\
& =(-1)^{\lceil\be\rceil}{}_{\cD'(0,T)}\langle
I^{\lceil\be\rceil-\be}
w,\chi^{(\lceil\be\rceil)}\rangle_{C^\infty_0(0,T)}\nonumber\\
& ={}_{\cD'(0,T)}\langle D_t^\be w,\chi\rangle_{C^\infty_0(0,T)}.\label{lll1f}
\end{align}
In the same way, we have
\begin{align*}
\lim_{k\to\infty}{}_{\cD'(0,T)}\langle w_k,\chi\rangle_{C^\infty_0(0,T)}
& ={}_{\cD'(0,T)}\langle w,\chi\rangle_{C^\infty_0(0,T)},\\
\lim_{k\to\infty}{}_{\cD'(0,T)}\langle h_k,\chi\rangle_{C^\infty_0(0,T)}
& ={}_{\cD'(0,T)}\langle h,\chi\rangle_{C^\infty_0(0,T)}.
\end{align*}
Combining these identities with \eqref{lll1f} and \eqref{lll1c},
we obtain
\begin{align*}
{}_{\cD'(0,T)}\langle h,\chi\rangle_{C^\infty_0(0,T)}
& =\lim_{k\to\infty}{}_{\cD'(0,T)}\langle h_k,\chi\rangle_{C^\infty_0(0,T)}
=\lim_{k\to\infty}{}_{\cD'(0,T)}\langle (D_t^\be+\la)w_k,
\chi\rangle_{C^\infty_0(0,T)}\\
& ={}_{\cD'(0,T)}\langle (D_t^\be+\la)w,\chi\rangle_{C^\infty_0(0,T)}.
\end{align*}
This completes the proof of the lemma.
\end{proof}

Now we are in a position to prove Theorem \ref{t2}.
%%%%%%%%%%%%%%%%%%%%

\begin{proof}[Proof of Theorem $\ref{t2}$]
Let $u$ be the solution of \eqref{eq1} with a constant $\al\in(0,2)$. Without
loss of generality, we only investigate the case of $\al\in(0,1)$ with
$\rho\equiv1$ because the case of $\al\in(1,2)$ with $\bm b\equiv\bm0$ can
be treated in the same manner. For clarity, we divide the proof into three
steps.\medskip

{\bf Step 1 } Let us fix $T\ge T_0$ and define
\begin{equation}\label{v}
v(t,\,\cdot\,):=I^{1-\al}u(t,\,\cdot\,),\quad t>0.
\end{equation}
Using the fact that $u\in L^1_\loc(\BR_+;L^2(\Om))$, we deduce that
$v\in L^1(0,T;L^2(\Om))$. In this step we will prove that actually
$v\in W^{1,1}(0,T;H^{-2}(\Om))$. Following Definition \ref{d1} of weak
solutions, we recall that for $p>p_0$, the Laplace transform
$\wh u(p;\,\cdot\,)$ of $u(t,\,\cdot\,)$ with respect to $t$ solves the following
boundary value problem 
\[
\begin{cases}
(p^\al+\rho^{-1}\cA)\wh u(p;\,\cdot\,)=-\rho^{-1}\bm b\cdot\nb\wh u(p;\,\cdot\,)
+\wh\mu(p)\rho^{-1}h & \mbox{in }\Om,\\
\cR\wh u(p;\,\cdot\,)=0 & \mbox{on }\pa\Om.
\end{cases}
\]
Therefore, for all $p>p_0$ and all $n\in\BN$, we have
\[
\left( (\rho^{-1}\cA+p^\al)\wh u(p;\,\cdot\,),\vp_n\right)_\rho
=-\left(\rho^{-1}\bm b\cdot\nb\wh u(p;\,\cdot\,),\vp_n\right)_\rho
+\wh\mu(p)\left(\rho^{-1}h,\vp_n\right)_\rho.
\]
Integrating by parts, we obtain
\[
\left((p^\al+\rho^{-1}\cA)\wh u(p;\,\cdot\,),\vp_n\right)_\rho
=\left(\wh u(p;\,\cdot\,),(p^\al+\rho^{-1}\cA)\vp_n\right)_\rho
=(\la_n+p^\al)\left(\wh u(p;\,\cdot\,),\vp_n\right)_\rho
\]
and it follows
\[
\left( \wh u(p;\,\cdot\,),\vp_n\right)_\rho
=\f{-\left(\rho^{-1}\bm b\cdot\nb\wh u(p;\,\cdot\,),\vp_n\right)_\rho}
{\la_n+p^\al}+\f{\wh\mu(p)\left(\rho^{-1}h,\vp_n\right)_\rho}{\la_n+p^\al}.
\]
In the same way, fixing
\[
w(t,\,\cdot\,):=\int_0^t S_2(t-\tau)\left[-\rho^{-1}\bm b\cdot\nb u(\tau,\,\cdot\,)
+\mu(\tau)\rho^{-1}h\right]\rd\tau,\quad t>0,
\]
we have
\[
\left( w(t,\,\cdot\,),\vp_n\right)_\rho=\int_0^t(t-\tau)^{\al-1}
E_{\al,\al}(-\la_n(t-\tau)^{\al})\left(-\rho^{-1}\bm b\cdot\nb u(\tau,\,\cdot\,)
+\mu(\tau)\rho^{-1}h,\vp_n\right)_\rho\rd\tau,\quad t>0
\]
for all $n\in\BN$. Therefore, applying the Laplace transform in
time for $p>p_0$ to the expression above, we get
\[
\left( \wh w(p;\,\cdot\,),\vp_n\right)_\rho
=\f{-\left(\rho^{-1}\bm b\cdot\nb\wh u(p;\,\cdot\,),\vp_n\right)_\rho}
{\la_n+p^\al}+\f{\wh\mu(p)\left(\rho^{-1}h,\vp_n\right)_\rho}{\la_n+p^\al}
=\left( \wh u(p;\,\cdot\,),\vp_n\right)_\rho
\]
for all $n\in\BN$. Then it follows that
\[
\wh w(p;\,\cdot\,)=\wh u(p;\,\cdot\,),\quad p>p_0,
\]
and the uniqueness of the Laplace transform in time implies that
\begin{equation}\label{idd1}
u(t,\,\cdot\,)=w(t,\,\cdot\,)=\int_0^t S_2(t-\tau)
\left[-\rho^{-1}\bm b\cdot\nb u(\tau,\,\cdot\,)+\mu(\tau)\rho^{-1}h\right]
\rd\tau,\quad t>0.
\end{equation}
In view of \eqref{idd1}, fixing
\begin{equation}\label{fn}
f_n(t)=\left(-\rho^{-1}\bm b\cdot\nb u(t,\,\cdot\,)+\mu(t)\rho^{-1}h,
\vp_n\right)_\rho,\quad n\in\BN,
\end{equation}
we deduce that
\begin{equation}\label{un}
u_n(t):=(u(t,\,\cdot\,),\vp_n)_\rho=\int_0^t(t-\tau)^{\al-1}
E_{\al,\al}(-\la_n(t-\tau)^{\al})f_n(\tau)\,\rd\tau,\quad n\in\BN,\ t>0.
\end{equation}

Since $u\in L^1(0,T;H^1(\Om))$ by \cite[pp.13--15]{Kia}, we
deduce that $-\rho^{-1}\bm b\cdot\nb u
+\mu\,\rho^{-1}h\in L^1(0,T;L^2(\Om))$ and thus
$f_n\in L^1(0,T)$ for all $n\in\BN$ by \eqref{fn}. Then in view of Lemma
\ref{lll1}, we know that
\begin{equation}\label{idd2}
D_t^\al u_n(t)=-\la_n u_n(t)+f_n(t),\quad n\in\BN,\ t\in(0,T)
\end{equation}
holds in the sense of distributions. Since
$-\rho^{-1}\bm b\cdot\nb u+\mu\,\rho^{-1}h
\in L^1(0,T;L^2(\Om))$, we deduce that the sequence 
\[
\sum_{n=1}^N f_n(t)\vp_n\quad(N\in\BN)
\]
converges in the sense of $L^1(0,T;L^2(\Om))$. In the same way, we can prove
that the sequence
\[
\sum_{n=1}^N\la_n u_n(t)\vp_n\quad(N\in\BN)
\]
converges in the sense of $L^1(0,T;\cD(A^{-1}))$, where we recall that
$A=\rho^{-1}\cA$ acts on $L^2(\Om;\rho\,\rd\bm x)$ with the boundary
condition $\cR u=0$ on $\pa\Om$. This proves that 
\[
\sum_{n=1}^ND_t^\al u_n(t)\vp_n=\sum_{n=1}^N[-\la_nu_n(t)+f_n(t)]
\vp_n\quad(N\in\BN)
\]
converges in the sense of $L^1(0,T;\cD(A^{-1}))$. Thus, $D_t^\al u$ is well
defined and
\[
D_t^\al u=\sum_{n=1}^\infty D_t^\al u_n(t)\vp_n\in L^1(0,T;\cD(A^{-1})).
\]
Let $H^2_0(\Om)$ be the closure of $C^\infty_0(\Om)$ in $H^2(\Om)$.
It is clear that $H^2_0(\Om)$ is embedded continuously into
$\cD(A)$ and by duality, we deduce that $\cD(A^{-1})$ embedded continuously
into $H^{-2}(\Om)$. It follows that $D_t^\al u\in L^1(0,T;H^{-2}(\Om))$. On the
other hand, the function $v$ given by \eqref{v} satisfies
$\pa_t v=D_t^\al u\in L^1(0,T;H^{-2}(\Om))$, and using the
fact that $v\in L^1(0,T;L^2(\Om))$, we get $v\in W^{1,1}(0,T;H^{-2}(\Om))$.\medskip

{\bf Step 2 } We fix $T_1\in[0,T)$. In this step, we will show that the condition 
\begin{equation}\label{tt2}
u=0\mbox{ in }(T_1,T)\times\om
\end{equation}
implies
\begin{equation}\label{t2d}
u=0\quad\mbox{in }(0,T)\times\om.
\end{equation}
Similarly to the proof of Theorem \ref{t1}, we set
$u_\psi(t):={}_{\cD'(\om)}\langle u(t,\,\cdot\,),\psi\rangle_{C^\infty_0(\om)}$
($0<t<T$) with arbitrarily fixed $\psi\in C^\infty_0(\om)$.
According to Step 1, the function $v$ defined by \eqref{v}
lies in $W^{1,1}(0,T;H^{-2}(\Om))$ and thus
\begin{equation}\label{idd4}
I^{1-\al}u_\psi={}_{\cD'(\om)}\langle v(t,\,\cdot\,),
\psi\rangle_{C^\infty_0(\om)}\in W^{1,1}(0,T).
\end{equation}
Moreover, in view of \eqref{idd2} and the fact that the sequence
\[
\sum_{n=1}^ND_t^\al u_n(t)\vp_n=\sum_{n=1}^N[-\la_nu_n(t)+f_n(t)]\vp_n,
\quad N\in\BN,
\]
converge in the sense of $L^1(0,T;\cD(A^{-1}))$ to $D_t^\al u$, we deduce that
for a.e.\! $t\in(0,T)$, we have
\begin{align*}
D_t^\al u(t,\,\cdot\,) & =\sum_{n=1}^\infty D_t^\al u_n(t)\vp_n
=\sum_{n=1}^\infty[-\la_nu_n(t)+f_n(t)]\vp_n\\
& =\sum_{n=1}^\infty\left(-\rho^{-1}\cA u(t,\,\cdot\,)
-\rho^{-1}\bm b\cdot\nb u(t,\,\cdot\,)+\mu(t)\rho^{-1}h,\vp_n\right)_\rho\vp_n\\
& =-\rho^{-1}\cA u(t,\,\cdot\,)-\rho^{-1}\bm b\cdot\nb u(t,\,\cdot\,)
+\mu(t)\rho^{-1}h.
\end{align*}
Here we recall that $f_n$ and $u_n$ were given by \eqref{fn}
and \eqref{un}, respectively. On the other hand, from \eqref{tt2} we deduce
that $-\rho^{-1}\cA u-\rho^{-1}\bm b\cdot\nb u=0$ in $(T_1,T)\times\om$.
Combining this with the fact that $h=0$ in $\om$, we obtain
\begin{equation}\label{idd5}
D_t^\al u=-\rho^{-1}\cA u-\rho^{-1}\bm b\cdot\nb u+\mu\,h=0
\quad\mbox{in }(T_1,T)\times\om.
\end{equation}
In the same way, using the fact that the function $v$ given by \eqref{v}
lies in $W^{1,1}(0,T;H^{-2}(\Om))$, we obtain
\[
\begin{aligned}
D_t^\al u_\psi(t) & =\f\rd{\rd t}I^{1-\al}u_\psi(t)=\f\rd{\rd t}
\left({}_{\cD'(\om)}\langle v(t,\,\cdot\,),\psi\rangle_{C^\infty_0(\om)}\right)\\
& ={}_{\cD'(\om)}\langle \pa_tv(t,\,\cdot\,),\psi\rangle_{C^\infty_0(\om)}
={}_{\cD'(\om)}\langle D_t^\al u(t,\,\cdot\,),\psi\rangle_{C^\infty_0(\om)}
\end{aligned}
\]
for a.e.\! $t\in(0,T)$, and \eqref{idd5} implies
$D_t^\al u_\psi=0$ in $(T_1,T)$. Therefore, we have
$u_\psi=D_t^\al u_\psi=0$ in $(T_1,T)$. Combining this with condition
\eqref{idd4} and applying \cite[Theorem 1]{JK} yields
\[
u_\psi(t)={}_{\cD'(\om)}\langle u(t,\,\cdot\,),\psi\rangle_{C^\infty_0(\om)}
=0,\quad0<t<T,
\]
which implies \eqref{t2d} since $\psi\in C^\infty_0(\om)$ was chosen
arbitrarily.\medskip

{\bf Step 3 } In this step, we show that \eqref{t2d} implies $h\equiv0$ in
$\Om$. We deal with the cases of $\al\in(0,1]$ and $\al\in(1,2)$ separately.\medskip

{\bf Case 1 } We first study the case of $\al\in(0,1]$ with $\rho\equiv1$. In the
same manner as before, for arbitrarily fixed $\psi\in C^\infty_0(\om)$, it
follows from \eqref{t2d} and the solution representation \eqref{t2cd} that
\[
0={}_{\cD'(\om)}\langle u(t,\,\cdot\,),\psi\rangle_{C_0^\infty(\om)}
=\int_0^t\mu(t-\tau)\,v_\psi(\tau)\,\rd\tau,\quad0<t<T,
\]
where $v_\psi(t):={}_{\cD'(\om)}\langle S_1(t)h,\psi\rangle_{C^\infty_0(\om)}
\in L^1(0,T)$ by Lemma \ref{l1}. Therefore, parallel to the proof of Theorem
\ref{t1}, the Titchmarsh convolution theorem guarantees the existence of
constants $\tau_1,\tau_2\in[0,T]$ satisfying $\tau_1+\tau_2\ge T$ such that
\[
\mu\equiv0\mbox{ in }(0,\tau_1),\quad v_\psi\equiv0\mbox{ in }(0,\tau_2).
\]
Since $\mu\not\equiv0$ in $(0,T_0)$ by the key assumption \eqref{t1b}, we
conclude $\tau_1<T_0$ and thus
\[
\tau_2\ge T-\tau_1>T-T_0\ge0.
\]
In other words, we obtain 
\[
v_\psi(t)={}_{\cD'(\om)}\langle S_1(t)h,\psi\rangle_{C^\infty_0(\om)}
=0,\quad0<t<\tau_2,
\]
which implies
\begin{equation}\label{t2f}
v(t,\,\cdot\,):=S_1(t)h=0\quad\mbox{in }\om,\ 0<t<\tau_2
\end{equation}
again since $\psi\in C^\infty_0(\om)$ was chosen arbitrarily. On the other
hand, similarly to the proofs of \cite[Theorem 1.1]{KSY} and
\cite[Theorem 1.4]{LKS}, we deduce that the map $t\longmapsto S_1(t)h$
admits an analytic extension from $\BR_+$ to $L^2(\Om)$. Therefore,
\eqref{t2f} implies
\begin{equation}\label{t2g}
v=0\quad\mbox{in }\BR_+\times\om.
\end{equation}
In view of Lemma \ref{l1}, the Laplace transform $\wh v(p;\,\cdot\,)$ of $v$
with respect to $t$ is well-defined for all $p>s_0$. Further, following \cite{Kia}
(see also \cite[Theorem 1.1]{KSY}), we see that $\wh v(p;\,\cdot\,)$ solves the
boundary value problem for an elliptic equation
\[
\begin{cases} 
(p^\al+\cL)\wh v(p;\,\cdot\,)=h & \mbox{in }\Om,\\
\cR\wh v(p;\,\cdot\,)=0 & \mbox{on }\pa\Om,
\end{cases}\quad\forall\,p>s_0.
\]
Meanwhile, \eqref{t2g} implies
\begin{equation}\label{t2h}
\wh v(p;\,\cdot\,)=0\quad\mbox{in }\om,\ p>s_0.
\end{equation}

On the other hand, let us introduce an initial-boundary value problem for a
parabolic equation
\[
\begin{cases} 
(\pa_t+\cL)w=0 & \mbox{in }\BR_+\times\Om,\\
w=h & \mbox{in }\{0\}\times\Om,\\
\cR w=0 & \mbox{in }\BR_+\times\pa\Om.
\end{cases}
\]
as well as the Laplace transform $\wh w(p;\,\cdot\,)$ of $w$ with respect to
$t$ for $p>p_0$, where $p_0\ge s_0^\al$ is sufficiently large. Then
straightforward calculation yields that for all $p>p_0$, there holds
$\wh v(p^{1/\al};\,\cdot\,)=\wh w(p;\,\cdot\,)$. Combining this with \eqref{t2h}
and applying the uniqueness of the Laplace transform, we obtain
\[
w=0\quad\mbox{in }\BR_+\times\om.
\]
Consequently, the unique continuation property of parabolic equations (e.g.
\cite[Theorem 1.1]{SS}) implies $w\equiv0$ in $\BR_+\times\Om$ and thus
$h=w(0,\,\cdot\,)\equiv0$ in $\Om$. This completes the proof of Theorem
\ref{t2} for $\al\in(0,1]$.\medskip

{\bf Case 2 } Now we turn to the case of $\al\in(1,2)$ with
$\bm b\equiv\bm0$. Combining the solution representation \eqref{t2dd} with
the above arguments, again we can conclude \eqref{t2f} with
$v(t,\,\cdot\,):=S_2(t)h$ in this case. On the other hand, applying again
\cite[Theorem 1.6]{P} and utilizing arguments similar to those used in the proof
of \cite[Proposition 3.1]{KLLY}, one can check that the map
$t\longmapsto v(t,\,\cdot\,)$ admits an analytic extension as a function from
$\BR_+$ to $L^2(\Om)$. Therefore, again \eqref{t2f} implies \eqref{t2g}.

Then for $p>0$, it follows from \cite[Theorem 1.1]{KSY} that the Laplace
transform $\wh v(p;\,\cdot\,)$ of $v$ with respect to $t$ solves the boundary
value problem for an elliptic equation
\[
\begin{cases}
(p^\al\rho+\cA)\wh v(p;\,\cdot\,)=h & \mbox{in }\Om,\\
\cR\wh v(p;\,\cdot\,)=0 & \mbox{on }\pa\Om.
\end{cases}
\]
Meanwhile, \eqref{t2g} implies \eqref{t2h} again. Therefore, repeating the
above arguments (see also Step 2 in the proof of \cite[Theorem 1.1]{KSXY}),
again we arrive at $h\equiv0$ in $\Om$. This completes the proof of Theorem
\ref{t2} for $\al\in(1,2)$.
\end{proof}
%%%%%%%%%%%%%%%%%%%%

\begin{proof}[Proof of Corollary $\ref{c2}$]
Since the proofs for the cases of $\al\in(0,1]$ and $\al\in(1,2)$ are again
similar, we only deal with the former one without loss of generality.

Introducing the auxiliary function $u:=u_1-u_2$, we see that $u$ satisfies
\begin{equation}\label{eq-c2-u}
\begin{cases}
(\pa_t^\al+\cL)u=F:=\mu_1h_1-\mu_2h_2 & \mbox{in }\BR_+\times\Om,\\
u=0 & \mbox{in }\{0\}\times\Om,\\
\cR u=0 & \mbox{on }\BR_+\times\pa\Om
\end{cases}
\end{equation}
along with the additional condition $u=0$ in $(T_1,T)\times\om$. Especially,
thanks to the assumption $T_1<T_0$ in \eqref{c2b}, we have
\begin{equation}\label{eq-c2-1}
u=0\quad\mbox{in }(T_1,T_0)\times\om. 
\end{equation}
On the other hand, it follows from \eqref{c2a}--\eqref{c2aa} that
\begin{equation}\label{eq-def-F}
F=\begin{cases}
\mu_1(h_1-h_2) & \mbox{in }(0,T_0)\times\Om,\\
(\mu_1-\mu_2)h_1 & \mbox{in }\BR_+\times\om
\end{cases}
\end{equation}
and in particular,
\begin{equation}\label{eq-c2-2}
F=0\quad\mbox{in }(0,T_0)\times\om.
\end{equation}
Therefore, repeating the argument used in Step 1 of the proof of Theorem
\ref{t2}, we utilize \eqref{eq-c2-1} and \eqref{eq-c2-2} to conclude $u\equiv0$
in $(0,T_0)\times\om$. Then a direct application of Theorem \ref{t2} implies
$h_1-h_2\equiv0$ or equivalently $h_1= h_2$ in $\Om$.

From now on we can write $F=\mu\,h_1$ in \eqref{eq-c2-u} with
$\mu:=\mu_1-\mu_2$ in $\BR_+\times\Om$ and it remains to show
$\mu\equiv0$ in $(0,T)$. Now that $u=0$ in $(0,T)\times\om$, we can take
advantage of the solution representation \eqref{t2cd} to obtain
\[
0={}_{\cD'(\om)}\langle u(t,\,\cdot\,),\psi\rangle_{C_0^\infty(\om)}
=\int_0^t\mu(t-\tau)\,v_\psi(\tau)\,\rd\tau,\quad0<t<T
\]
with arbitrarily fixed $\psi\in C_0^\infty(\om)$, where
$v_\psi(t):={}_{\cD'(\om)}\langle S_1(t)h_1,\psi\rangle_{C^\infty_0(\om)}$. Again
by the Titchmarsh convolution theorem, there exist constants
$\tau_1,\tau_2\in[0,T]$ satisfying $\tau_1+\tau_2\ge T$ such that
\[
\mu\equiv0\mbox{ in }(0,\tau_1),\quad v_\psi\equiv0\mbox{ in }(0,\tau_2).
\]
Now it suffices to show $\tau_2=0$ by contradiction. If $\tau_2>0$, then we
obtain $S_1(t)h_1=0$ in $\om$ for $0<t<\tau_2$. By the same argument used
in Step 2 of the proof of Theorem \ref{t2}, we can conclude $h_1=0$ in
$\Om$, which contradicts with the assumption $h_1\not\equiv0$. Then there
should hold $\tau_2=0$ and thus $T\ge\tau_1\ge T-\tau_2=T$, that is,
$\mu\equiv0$ in $(0,T)$. This completes the proof of Corollary \ref{c2}.
\end{proof}
%%%%%%%%%%%%%%%%%%%%%%%%%%%%%%%%%%%%%%%%

\Section{Proof of Theorem \ref{t3} and Corollary \ref{c3}}\label{sec-t3}

In this section, we assume that $\bm b\equiv\bm0$ and
$\al\in L^\infty(\Om;(0,1))$ fulfills \eqref{bb}--\eqref{vo}. We fix
$\te\in(\f\pi2,\pi)$, $\de>0$ and we define the contour in $\BC$, 
\[%begin{equation}\label{g1}
\ga(\de,\te):=\ga_-(\de,\te)\cup\ga_0(\de,\te)\cup\ga_+(\de,\te),
\]%end{equation}
oriented in the counterclockwise direction with
\begin{equation}\label{g2}
\ga_0(\de,\te):=\{\de\,\e^{\ri\,\be}\mid\be\in[-\te,\te]\},\quad\ga_\pm(\de,\te)
:=\{s\,\e^{\pm\ri\te}\mid s\in[\de,\infty)\},
\end{equation}
with double signs in the same order. Then, following
\cite{KSY}, we define the operator
\begin{equation}\label{S}
S(t)\psi:=\f1{2\pi\,\ri}\int_{\ga(\de,\te)}\e^{t p}(p^\al\rho+A)^{-1}\psi\,\rd p,
\quad t>0
\end{equation}
Recall that here $A=\rho^{-1}\cA$ acts on $L^2(\Om;\rho\,\rd\bm x)$, with the
boundary condition $\cR u=0$ in $\pa\Om$. Moreover, according to
\cite[Theorem 1.1]{KSY}, the operator $S$ is independent of the choice of
$\te\in\left(\f\pi2,\pi\right)$, $\de>0$. In light of \cite[Theorem 1.1]{KSY} and
\cite[Remark 1]{KSY}, for $\mu\in L^\infty(\BR_+)$ compactly supported,
problem \eqref{eq1} admits a unique weak solution (in the sense of Definition
\ref{d1}) given by
\[
u(t,\,\cdot\,)=
\int_0^t\mu(t-\tau)\,S(\tau)(\rho^{-1}h)\,\rd\tau,\quad t>0.
\]
We prove in Proposition \ref{pp2} that the result of \cite{KSY} can be extended
to problem \eqref{eq1} with $\mu\in L^1(\BR_+)$ compactly supported. Using
the representation \eqref{S} of the weak solution of \eqref{eq1}, we are now in
position to complete the proof of Theorem \ref{t3} and Corollary \ref{c3}.

\begin{proof}[Proof of Theorem $\ref{t3}$]
Let $u$ be the solution of \eqref{eq1} with a variable $\al$ satisfying
\eqref{bb}--\eqref{vo} and $\om\subset\Om$ satisfy \eqref{t3aa}. Similarly to
the proof of Theorem \ref{t2}, we divide the proof into two steps.\medskip

{\bf Step 1 } First we prove that, for any $T\ge T_0$ and $T_1\in[0,T)$, the
condition \eqref{tt2} implies \eqref{t2d}. For this purpose, we investigate each
subdomain $\Om_\ell$ ($\ell=1,\ldots,N$), on which $\al(\bm x)=\al_\ell$
reduces to a constant. If $\om\cap\Om_\ell\ne\emptyset$, then obviously
$\cA u=0$ in $(T_1,T)\times(\om\cap\Om_\ell)$. Meanwhile, the assumption
\eqref{t3aa} implies $h=0$ in $\om\cap\Om_\ell$. Then in view of the
governing equation in \eqref{eq1}, we obtain
\[
\pa_t^{\al_\ell}u=-\rho^{-1}(\cA u+\mu\,h)=0
\quad\mbox{in }(T_1,T)\times(\om\cap\Om_\ell).
\]
Repeating the argument used in Step 1 of the proof of Theorem \ref{t2}, we can
deduce $u=0$ in $(0,T)\times(\om\cap\Om_\ell)$, which indicates \eqref{t2d}
by collecting all nonempty $\om\cap\Om_\ell$ ($\ell=1,\ldots,N$).\medskip

{\bf Step 2 } We show that \eqref{t2d} implies $h\equiv0$ in $\Om$ in this
step. According to Proposition \ref{pp2}, $u$ takes the form 
\[
u(t,\,\cdot\,)=\int_0^t\mu(t-\tau)\,v(\tau,\,\cdot\,)\,\rd\tau
\]
with $v(t,\,\cdot\,):=S(t)\rho^{-1}h$. Applying condition \eqref{vo} and Lemma
\ref{lll2}, we deduce that $v\in L^1_\loc(\BR_+;L^2(\Om))$ with
\begin{equation}\label{ess2}
\|v(t,\,\cdot\,)\|_{L^2(\Om)}\le C\|\rho^{-1}h\|_{L^2(\Om)}
\max\left(t^{2\al_N-\al_1-1},t^{2\al_1-\al_N-1},1\right),\quad t>0.
\end{equation}
Combining this result with the arguments of Step 2 of the proof of Theorem
\ref{t2}, we can show that \eqref{t2d} implies the existence of a constant
$\tau_2>0$ such that
\begin{equation}\label{t3f}
v=0\quad\mbox{in }(0,\tau_2)\times\om.
\end{equation}
Following \cite[Theorem 1.1]{KSY}, we deduce that the map
$t\longmapsto v(t,\,\cdot\,)$ admits an analytic extension from $\BR_+$ to
$L^2(\Om)$. Therefore, the condition \eqref{t3f} implies
\begin{equation}\label{t3g}
v=0\quad\mbox{in }\BR_+\times\om.
\end{equation}
On the other hand, the estimate \eqref{ess2} implies
\[
\int_0^\infty\e^{-p t}\|v(t,\,\cdot\,)\|_{L^2(\Om)}\,\rd t<\infty,\quad p>0
\]
and applying the properties of the operator $S$ borrowed form
\cite[Theorem 1.1]{KSY}, we obtain
\[
\wh v(p;\,\cdot\,)=(p^\al+A)^{-1}(\rho^{-1}h),\quad p>0.
\]
Thus, $\wh v(p;\,\cdot\,)$ solves the boundary value problem for an elliptic
equation
\[
\begin{cases}
(p^\al\rho+\cA)\wh v(p;\,\cdot\,)=h & \mbox{in }\Om,\\
\cR\wh v(p;\,\cdot\,)=0 & \mbox{on }\pa\Om.
\end{cases}
\]
Simultaneously, \eqref{t3g} implies that
\[
\wh v(p;\,\cdot\,)=0\quad\mbox{in }\om,\ p>0.
\]
Combining this with \eqref{t3b} and \eqref{t3aa}, we deduce that for all $p>0$,
the restriction of $\wh v(p;\,\cdot\,)$ to $\cO$ satisfy
\begin{equation}\label{eq-gov-v}
\begin{cases}
(p^\al\rho+\cA)\wh v(p;\,\cdot\,)=0 & \mbox{in }\cO,\\
\wh v(p;\,\cdot\,)=0 & \mbox{in }\om\cap\cO. 
\end{cases}
\end{equation}
Therefore, applying the unique continuation property of elliptic equations (see
e.g. \cite[Theorem 1]{SS1}) yields
\begin{equation}\label{eq-v-0}
\wh v(p;\,\cdot\,)=0\quad\mbox{in }\cO,\ p>0.
\end{equation}
This, together with the fact that $\cR\wh v(p;\,\cdot\,)=0$ on $\pa\Om$ and
$\pa\Om_\ell\subset\pa\Om\cup\cO$ for all $\ell=1,\ldots,N$ by \eqref{t3b},
leads to
\[
\cR\wh v(p;\,\cdot\,)=0\quad\mbox{on }\pa\Om_\ell,\ p>0,\ \ell=1,\ldots,N.
\]
Therefore, combining this boundary condition with \eqref{eq-gov-v} and
\eqref{eq-v-0}, we find that the restriction of $\wh v(p;\,\cdot\,)$ to $\Om_\ell$
($\ell=1,\ldots,N$) satisfies
\[
\begin{cases}
(p^{\al_\ell}\rho+\cA)\wh v(p;\,\cdot\,)=h & \mbox{in }\Om_\ell,\\
\cR\wh v(p;\,\cdot\,)=0 & \mbox{on }\pa\Om_\ell,\\
\wh v(p;\,\cdot\,)=0 & \mbox{in }\Om_\ell\cap\cO,
\end{cases}\quad p>0.
\]
Here we notice that $\Om_\ell\cap\cO\ne\emptyset$ by \eqref{t3b}. Thus,
utilizing the arguments used in Step 3 of the proof of
\cite[Theorem 1.1]{KSXY}, we conclude $h=0$ in $\Om_\ell$ for all
$\ell=1,\ldots,N$ and consequently $h\equiv0$. This completes the proof of
Theorem \ref{t3}.
\end{proof}

\begin{proof}[Proof of Corollary $\ref{c3}$]
Similarly to the proof of Corollary \ref{c2}, the auxiliary function $u:=u_1-u_2$
satisfies
\[
\begin{cases}
(\rho\,\pa_t^\al+\cA)u=F:=\mu_1h_1-\mu_2h_2
& \mbox{in }\BR_+\times\Om,\\
u=0 & \mbox{in }\{0\}\times\Om,\\
\cR u=0 & \mbox{on }\BR_+\times\pa\Om
\end{cases}
\]
with $u=0$ in $(T_1,T)\times\om$, where $F$ satisfies \eqref{eq-def-F}.
Especially, we have $F=\mu_1(h_1-h_2)$ in $(0,T_0)\times\om$, where
$h_1-h_2=0$ in $\om\cup\cO$ by \eqref{eq-assume-h}. Repeating the
argument used in Step 1 of the proof of Theorem \ref{t3}, we can conclude
$u=0$ in $(0,T_0)\times\om$, which further implies $h_1-h_2\equiv0$ or
equivalently $h_1= h_2$ in $\Om$ by Theorem \ref{t3}.

Finally, in a similar way to the end of the proof of Corollary \ref{c2}, we arrive
at $\mu_1\equiv\mu_2$ in $(0,T)$, which completes the proof of Corollary
\ref{c3}.
\end{proof}
%%%%%%%%%%%%%%%%%%%%%%%%%%%%%%%%%%%%%%%%

\Section{Appendix}\label{sec-app}

Let $\al\in L^\infty(\Om)$ be such that there exist two constants
$\al_0,\al_M\in(0,1)$ such that
\begin{equation}\label{alpha1}
0<\al_0\le\al\le\al_M<1,\quad\al_M<2\al_0.
\end{equation}
In this section, we prove the unique existence of a weak
solution to the problem \eqref{eq1} when $\bm b\equiv\bm0$,
$\al\in L^\infty(\Om)$ satisfies \eqref{alpha1} and $\mu\in L^1(\BR_+)$ is
compactly supported. We start with the following intermediate result.

\begin{lem}\label{lll2}
Let $\te\in\left(\f\pi2,\pi\right)$. The map $t\longmapsto S(t)$ defined by
\eqref{S} lies in $L^1_\loc(\BR_+;\cB(L^2(\Om))$ and there
exists a constant $C>0$ depending only on $\cA,\rho,\te,\Om$
such that
\begin{equation}\label{ll2a}
\|S(t)\|_{\cB(L^2(\Om))}\le
C\max\left(t^{2\al_M-\al_0-1},t^{2\al_0-\al_M-1},1\right),\quad t>0.
\end{equation}
\end{lem}

\begin{proof} 
Throughout this proof, by $C>0$ we denote generic constants
depending only on $\cA,\rho,\te,\Om$, which may change from line to line.
In light of \cite[Proposition 2.1]{KSY}, for all $\be\in(0,\pi)$, we have
\begin{equation}\label{ll2b}
\left\|\left(A+(r\,\e^{\ri\be_1})^\al\right)^{-1}\right\|_{\cB(L^2(\Om ))}
\le C\max\left(r^{\al_0-2\al_M},r^{\al_M-2\al_0}\right),
\quad r>0,\ \be_1\in(-\be,\be),
\end{equation}
where $C>0$ depends only on $\cA,\rho,\Om$ and $\be$.
Using the fact that the operator $S$ is independent of the choice of $\de>0$
(see the discussion at the beginning of Section \ref{sec-t3}), we can
decompose
\[
S(t)=S_-(t)+S_0(t)+S_+(t),\quad t>0,
\]
where
\[
S_m(t)=\f1{2\pi\,\ri}\int_{\ga_m(t^{-1},\te)}\e^{t p}(p^\al+A)^{-1}\,\rd p,
\quad m=0,\mp,\ t>0.
\]
Here we recall that the contours $\ga_m(\de,\te)$ for
$m=0,\mp$ and $\de>0$ were given by \eqref{g2}.

In order to complete the proof of the lemma, it suffices to prove
\begin{equation}\label{ll2c}
\|S_m(t)\|_{B(L^2(\Om)}\le C\max\left(t^{2\al_M-\al_0-1},
t^{2\al_0-\al_M-1},1\right),\quad t>0,\ m=0,\mp.
\end{equation}
Indeed, it follows from \eqref{alpha1} that $2\al_0-\al_M-1>-1$
and $2\al_M-\al_0-1>\al_M-1>-1$. Thus, applying the estimates \eqref{ll2c}, we
can easily deduce that $S\in L^1_\loc(\BR_+;\cB(L^2(\Om))$ satisfies the
estimate \eqref{ll2a}.

For $m=0$, using \eqref{ll2b}, we find
\[
\|S_0(t)\|_{\cB(L^2(\Om)}\le C\int_{-\te}^\te
t^{-1}\left\|\left(A+(t^{-1}\e^{\ri\be})^\al\right)^{-1}\right\|_{\cB(L^2(\Om)}
\rd\be\le C\max\left(t^{2\al_M-\al_0-1},t^{2\al_0-\al_M-1}\right),
\]
which implies \eqref{ll2c} for $m=0$. For $m=\mp$, again we
employ \eqref{ll2b} to estimate
\begin{align*}
\|S_\mp(t)\|_{\cB(L^2(\Om)} & \le C\int_{t^{-1}}^\infty\e^{r t\cos\te}
\left\|\left(A+(r\,\e^{\ri\te})^\al\right)^{-1}\right\|_{\cB(L^2(\Om)}\rd r\\
& \le C\int_{t^{-1}}^\infty\e^{r t\cos\te}
\max\left(r^{\al_0-2\al_M},r^{\al_M-2\al_0}\right)\rd r.
\end{align*}
For $t>1$, we obtain
\begin{align*}
\|S_\mp(t)\|_{\cB(L^2(\Om)} & \le C\int_1^\infty\e^{r t\cos\te}
r^{\al_M-2\al_0}\,\rd r+C\int_{t^{-1}}^1r^{\al_0-2\al_M}\,\rd r\\
& \le C\int_0^\infty\e^{r t\cos\te}r^{\al_M-2\al_0}\,\rd r
+C\left(t^{2\al_M-\al_0-1}+1\right)\\
& \le C\,t^{-1}\int_0^\infty\e^{r\cos\te}
\left(\f r t\right)^{\al_M-2\al_0}\rd r
+C\left(t^{2\al_M-\al_0-1}+1\right)\\
& \le C\max\left(t^{2\al_0-\al_M-1},t^{2\al_M-\al_0-1},1\right).
\end{align*}
In the same way, for $t\in(0,1]$, we get
\[
\|S_\mp(t)\|_{\cB(L^2(\Om)}\le C\int_1^\infty\e^{r t\cos\te}
r^{\al_M-2\al_0}\,\rd r\le C\,t^{2\al_0-\al_M-1}.
\]
Combining these two estimates, we obtain
\[
\|S_\mp(t)\|_{\cB(L^2(\Om)}\le
C\max\left(t^{2\al_0-\al_M-1},t^{2\al_M-\al_0-1},1\right),\quad t>0.
\]
This proves that \eqref{ll2c} also holds true for $m=\mp$. This
completes the proof of the lemma.
\end{proof}
%%%%%%%%%%%%%%%%%%%%

Using Lemma \ref{lll2}, we will show that
\cite[Theorem 1.1]{KSY} can be extended to $\mu\in L^1(\BR_+)$
with a compact support.

\begin{prop}\label{pp2} 
Let $h\in L^2(\Om),$ $\al\in L^\infty(\Om)$ satisfy \eqref{alpha1} and
$\mu\in L^1(\BR_+)$ be compactly supported. Then there exists a unique weak
solution $u\in L^1_\loc(0,T;L^2(\Om))$ to \eqref{eq1}.
\end{prop}

\begin{proof}
Suppose that $\supp\,\mu\subset[0,T]$ for some $T>0$ and
choose a sequence $\{\mu_n\}_{n\in\BN}\subset C^\infty_0(0,T+1)$ satisfying
\begin{equation}\label{p2a}
\lim_{n\to\infty}\|\mu_n-\mu\|_{L^1(\BR_+)}=0.
\end{equation}
In light of Lemma \ref{lll2}, for $t\ge0$ we can introduce
\begin{align*}
u_n(t,\,\cdot\,) & :=\int_0^t\mu_n(t-\tau)\,S(\tau)(\rho^{-1}h)\,\rd\tau,
\quad n\in\BN,\\
u(t,\,\cdot\,) & :=\int_0^t\mu(t-\tau)\,S(\tau)(\rho^{-1}h)\,\rd\tau
\end{align*}
as elements of $L^1_\loc(\BR_+;L^2(\Om))$. We prove that for all $p>0$, the
Laplace transform $\wh u(p)$ of $u$ is well-defined in $L^2(\Om)$ and we
have
\begin{equation}\label{p2c}
\lim_{n\to\infty}\|\wh{u_n}(p)-\wh u(p)\|_{L^2(\Om)}=0.
\end{equation}
Applying estimate \eqref{ll2a}, we obtain
\begin{align*}
\left\|\e^{-p t}u(t,\,\cdot\,)\right\|_{L^2(\Om)} & \le\int_0^t
\e^{-p\tau}\|S(\tau)\|_{\cB(L^2(\Om))}\|\rho^{-1}h\|_{L^2(\Om)}\,
\e^{-p(t-\tau)}|\mu(t-\tau)|\,\rd\tau\\
& \le C\left(\e^{-p t}\max
\left(t^{2\al_0-\al_M-1},t^{2\al_M-\al_0-1},1\right)\right)*
\left(\e^{-p t}|\mu(t)|\right)
\end{align*}
for all $t>0$ and $p>0$, where $*$ denotes the convolution in
$\BR_+$. Therefore, applying Young's convolution inequality
and condition \eqref{alpha1}, we deduce
\begin{align*}
\|\wh u(p)\|_{L^2(\Om)} & \le\int_0^\infty
\left\|\e^{-p t}u(t,\,\cdot\,)\right\|_{L^2(\Om)}\rd t\\
& \le C\left(\int_0^\infty\e^{-p t}
\max\left(t^{2\al_0-\al_M-1},t^{2\al_M-\al_0-1},1\right)\rd t\right)
\left(\int_0^\infty\e^{-p t}|\mu(t)|\,\rd t\right)<\infty.
\end{align*}
for all $p>0$. This proves the well-posedness of $\wh u(p)$ for
all $p>0$ in the sense of $L^2(\Om)$. In the same way, for all
$t>0$, $p>0$ and $n\in\BN$, we get
\begin{align*}
\left\|\e^{-p t}(u_n-u)(t,\,\cdot\,)\right\|_{L^2(\Om)}
& \le\int_0^t\e^{-p\tau}\|S(\tau)\|_{\cB(L^2(\Om))}\|\rho^{-1}h\|_{L^2(\Om)}\,
\e^{-p(t-\tau)}|(\mu_n-\mu)(t-\tau)|\,\rd\tau\\
& \le C\left(\e^{-p t}
\max\left(t^{2\al_0-\al_M-1},t^{2\al_M-\al_0-1},1\right)\right)*
\left(\e^{-p t}|(\mu_n-\mu)(t)|\right).
\end{align*}
Thus, applying Young's convolution inequality again, we have
\begin{align*}
& \quad\,\|\wh{u_n}(p)-\wh u(p)\|_{L^2(\Om)}
\le\int_0^\infty\left\|\e^{-p t}(u_n-u)(t,\,\cdot\,)\right\|_{L^2(\Om)}\rd t\\
& \le C\left(\int_0^\infty\e^{-p t}
\max\left(t^{2\al_0-\al_M-1},t^{2\al_M-\al_0-1},1\right)\rd t\right)
\left(\int_0^\infty\e^{-p t}|\mu_n(t)-\mu(t)|\,\rd t\right)\\
& \le C\|\mu_n-\mu\|_{L^1(\BR_+)}
\end{align*}
for all $p>0$ and $n\in\BN$, and \eqref{p2a} implies \eqref{p2c}. On the other
hand, in view of \cite[Theorem 1.1 and Remark 1]{KSY}, since
$\mu_n\in L^\infty(0,T)$, we have
\[
(p^\al+A)\wh{u_n}(p)=\left(\int_0^\infty\e^{-p t}\mu_n(t,\,\cdot\,)\,\rd t\right)
\rho^{-1}h,\quad p>0
\]
for all $n\in\BN$. In addition, \eqref{p2a} implies that
\[
\lim_{n\to\infty}\wh{\mu_n}(p)=\wh\mu(p),\quad p>0.
\]
Therefore, we obtain
\[
\lim_{n\to\infty}
\left\|\wh{u_n}(p)-\wh\mu(p)(A+p^\al)^{-1}(\rho^{-1}h)\right\|_{L^2(\Om)}=0
\]
and \eqref{p2c} implies that $\wh u(p)=(A+p^\al)^{-1}\wh {\mu}(p)\rho^{-1}h$,
$p>0$. From the definition of the operator $A$, we deduce that
$\wh u(p)$ solves the boundary value problem \eqref{d1a} for
all $p>0$. Therefore, we conclude that $u$ is the unique weak
solution of \eqref{eq1} and the proof is completed.
\end{proof}\medskip
%%%%%%%%%%%%%%%%%%%%%%%%%%%%%%%%%%%%%%%%

{\bf Acknowledgement}\ \ This paper has been supported by
Grant-in-Aid for Scientific Research (S) 15H05740, Japan Society for the
Promotion of Science (JSPS). The work of Y.\! Kian is partially supported by
the French National Research Agency ANR (project MultiOnde) grant ANR-17
CE40-0029. Y.\! Liu is supported by Grant-in-Aid for Early
Career Scientists 20K14355, JSPS. M.\! Yamamoto is supported by Grant-in-Aid
for Scientific Research (A) 20H00117, JSPS and by the National Natural
Science Foundation of China (Nos.\! 11771270, 91730303).

%%%%%%%%%%%%%%%%%%%%%%%%%%%%%%%%%%%%%%%%
% ---- Bibliography ----

\end{document}